\documentclass[a4paper,12pt]{article}
\usepackage{amssymb}
\usepackage{amsmath}
\usepackage{MnSymbol}
\usepackage{amscd}
\usepackage{eucal}
\usepackage{amsthm}
\usepackage{enumerate}
\usepackage{tikz-cd}

\newif\ifshowcomments

\showcommentsfalse

\newtheorem{theorem}{Theorem}[section]

\newtheorem{lemma}[theorem]{Lemma}
\newtheorem{proposition}[theorem]{Proposition}

\newtheorem{letterthm}{Theorem}

\newtheorem*{introqu}{Question}
\newtheorem*{introconj}{Conjecture}

\makeatletter
\newtheorem*{rep@theorem}{\rep@title}
\newcommand{\newreptheorem}[2]{%
\newenvironment{rep#1}[1]{%
 \def\rep@title{#2 \ref{##1}}%
 \begin{rep@theorem}}%
 {\end{rep@theorem}}}
\makeatother

\newreptheorem{theorem}{Theorem}
\newreptheorem{cor}{Corollary}

\theoremstyle{definition}
\newtheorem{definition}[theorem]{Definition}

\theoremstyle{remark}
\newtheorem{remark}[theorem]{Remark}
\newtheorem{example}[theorem]{Example}

\newcommand{\C}{\mathbb{C}}
\newcommand{\Z}{\mathbb{Z}}

\DeclareMathOperator{\Aut}{Aut}

\DeclareMathOperator{\Mod}{Mod}
\DeclareMathOperator{\Out}{Out}
\DeclareMathOperator{\PMod}{PMod}

\newcommand{\curlyO}{\mathcal{O}}
\newcommand{\curlyP}{\mathcal{P}}

\newcommand{\into}{\hookrightarrow}

\ifshowcomments
\newcommand{\comment}[1]{\marginpar{\sffamily{\tiny #1
\par}\normalfont}}
\else
\newcommand{\comment}[1]{}
\fi

\usepackage{hyperref}

\input xy
\xyoption{all}

\title{The congruence subgroup property for mapping class groups and the residual finiteness of hyperbolic groups}
\author{Henry Wilton\\ \emph{with an appendix by} Alessandro Sisto}

\newcommand{\Addresses}{{
  \bigskip
  \footnotesize

  \textsc{DPMMS, Centre for Mathematical Sciences, Wilberforce Road, Cambridge, CB3 0WB, UK}\par\nopagebreak
  \textit{E-mail address:} \texttt{h.wilton@maths.cam.ac.uk}\\[1ex]
  
  \textsc{Department of Mathematics, Heriot-Watt University and Maxwell Institute for Mathematical Sciences, Edinburgh, UK}\par\nopagebreak
\textit{E-mail address:} \texttt{a.sisto@hw.ac.uk}
}}

\begin{document}

\maketitle

\begin{abstract}
Assuming that every hyperbolic group is residually finite, we prove the congruence subgroup property for mapping class groups of hyperbolic surfaces of finite type. Under the same assumption, two consequences for profinite rigidity follow: profinitely equivalent hyperbolic 3-manifolds are commensurable; and fibred hyperbolic 3-manifolds are profinitely rigid if the monodromy is weakly asymmetric.
\end{abstract}

Our purpose is to relate a notorious open problem in geometric group theory to a different notorious problem in low-dimensional topology. The former is the problem of the residual finiteness of hyperbolic groups \cite[Question 1.15]{bestvina_questions_????}. 

\begin{introqu}[Residual finiteness of hyperbolic groups]
Is every word-hyperbolic group residually finite?
\end{introqu}
\noindent A remarkable positive result was proved by Agol, who answered the question in the affirmative for cubulable hyperbolic groups \cite{agol_virtual_2013}, building on results of Wise \cite{wise_structure_2021}. Nevertheless, in full generality, the question remains wide open.

The question in low-dimensional topology is known as the \emph{congruence subgroup property} for mapping class groups, which was conjectured to hold by Ivanov \cite{ivanov_fifteen_2006} (see also \cite[Problem 2.10]{kirby_problems_1995}, \cite[\S6]{lochak_results_2012}). A 2-dimensional surface $\Sigma$ is said to be \emph{of finite type} if $\Sigma$ is obtained by deleting finitely many points from a closed surface.

\begin{introconj}[Congruence subgroup property for mapping class groups]
Let $\Sigma$ be a connected, orientable, hyperbolic surface of finite type. Every subgroup of finite index in the mapping class group $\Mod(\Sigma)$ is a congruence subgroup.
\end{introconj}
Intuitively, the congruence subgroups of $\Mod(\Sigma)$ are those induced from the finite-sheeted covering spaces of $\Sigma$ -- see \S\ref{sec: CSP preliminaries} for a precise definition.  The congruence subgroup property for mapping class groups was proved in the case of surfaces of genus zero by Diaz, Donagi, and Harbater  \cite{diaz_every_1989}, surfaces of genus one by Asada \cite{asada_faithfulness_2001}, and surfaces of genus two by Boggi \cite{boggi_congruence_2009,boggi_congruence_2020}; see also \cite{ellenberg_arithmetic_2012,kent_congruence_2016,mcreynolds_congruence_2012}.

Our main theorem asserts that the conjecture is true, if every hyperbolic group is residually finite.

\begin{letterthm}\label{thm: Main theorem}
Let $\Sigma$ be a connected, orientable, hyperbolic surface of finite type. If every word-hyperbolic group is residually finite then every subgroup of finite index in $\Mod(\Sigma)$ is congruence.
\end{letterthm}

Among many possible applications, the congruence subgroup property plays an important role in the field of \emph{profinite rigidity} -- the study of the extent to which a group is determined by its finite quotients. Bridson and Reid posed one of the principal open questions in that area.  Let $\widehat{G}$ denote the profinite completion of a group $G$.

\begin{introconj}[Bridson--Reid conjecture]
Let $G$ and $H$ be lattices in $PSL_2(\C)$. If $\widehat{G}\cong\widehat{H}$ then $G\cong H$.\footnote{In fact, Bridson and Reid make the much stronger conjecture that the same holds if $H$ is any finitely generated, residually finite group \cite[Conjecture 2.1]{bridson_profinite_2023}. That stronger conjecture will not concern us here.}
\end{introconj}
\noindent Liu proved that hyperbolic 3-manifolds are profinitely rigid \emph{up to finite ambiguity} \cite{liu_finite-volume_2023} (see also \cite{liu_mapping_2023}). The conjecture has been proved for certain examples, but remains open in general \cite{boileau_profinite_2020,bridson_absolute_2020,bridson_profinite_2020b,bridson_profinite_2017,cheetham-west_profinite_2024,cheetham-west_absolute_????}.

Combining Theorem \ref{thm: Main theorem} with recent developments in the study of profinite rigidity for Kleinian groups, we obtain two partial result towards the Bridson--Reid conjecture.  If every hyperbolic group is residually finite then, on the one hand, profinitely equivalent hyperbolic 3-manifolds are commensurable, and on the other hand, mapping tori of \emph{weakly asymmetric} pseudo-Anosovs are profinitely rigid.

\begin{letterthm}\label{thm: Profinite rigidity}
Suppose that every hyperbolic group is residually finite. Let $G$ and $H$ be lattices in $PSL_2(\C)$.
\begin{enumerate}[(i)]
\item If $\widehat{G}\cong\widehat{H}$ then $G$ and $H$ are commensurable.
\item If, in addition, $G$ is the fundamental group of a fibred hyperbolic 3-manifold $M_\phi$, where the monodromy $\phi$ is a weakly asymmetric pseudo-Anosov, then $G\cong H$.
\end{enumerate}
\end{letterthm}

A pseudo-Anosov mapping class of a surface $\Sigma$ is called (weakly) asymmetric if it does not descend along a proper (normal) covering map to an orbifold quotient of $\Sigma$ (see Definition \ref{def: Asymmetric pseudo-Anosov} below). Masai showed that, aside from low-complexity cases that admit a hyperelliptic involution, a random walk on $\Mod(\Sigma)$ lands on an asymmetric pseudo-Anosov with probability tending to 1 \cite[Theorem 1.3]{masai_fibered_2017}. Therefore, in a suitable sense, item (ii) of Theorem \ref{thm: Profinite rigidity} asserts that the Bridson--Reid conjecture holds for generic fibred 3-manifolds.

Since it is widely believed that there is a non-residually finite hyperbolic group, the sceptical reader may wonder how useful Theorem \ref{thm: Main theorem} is.  We would make two cases for it.

On the one hand, optimists may hope to upgrade Theorem \ref{thm: Main theorem} to an unconditional proof of the congruence subgroup property.  They may be inspired in this endeavour by a theorem of Agol--Groves--Manning \cite{agol_residual_2009}. Combined with Kahn--Markovic's surface subgroup theorem \cite{kahn_immersing_2012}, the Agol--Groves--Manning theorem gave a conditional proof of the virtual Haken conjecture, and their techniques played a central role in Agol's eventual unconditional proof \cite{agol_virtual_2013} (relying, of course, on the transformative work of Wise \cite{wise_structure_2021}). The techniques in question are the combinatorial Dehn filling results of Groves--Manning \cite{groves_dehn_2008} and Osin \cite{osin_peripheral_2007}, and combinatorial Dehn filling, further refined to the setting of hierachically hyperbolic groups, likewise plays a central role in the proofs of this paper. 

In any case, we emphasise that our proof of the congruence subgroup property does not rely on \emph{all} hyperbolic groups being residually finite. Rather, only needs residual finiteness for the quotients of mapping class groups that arise in the proof of Theorem \ref{thm: HHG theorem}.

On the other hand, pessimists may prefer to try to use Theorem \ref{thm: Main theorem} to show that there is a non-residually finite hyperbolic group. Examples of non-homeomorphic but profinitely equivalent $\mathrm{Sol}$-manifolds \cite{funar_torus_2013} and Seifert-fibred spaces \cite{hempel_some_2014,wilkes_profinite_2017c} are known. Such pessimists may hope to develop hyperbolic analogues of these constructions, thereby disproving the Bridson--Reid conjecture. If the counterexamples are strong enough to contradict the conclusions of Theorem \ref{thm: Profinite rigidity}, then it would follow that a non-residually finite hyperbolic group exists. Such constructions would not disprove the congruence subgroup property, because  the residual finiteness of hyperbolic groups is used a second time in the proof of Theorem \ref{thm: Profinite rigidity}.

\subsection*{Sketch proof}

The proof of Theorem \ref{thm: Main theorem} uses Wise's notion of \emph{omnipotence}. For an independent set of infinite-order elements $\{\gamma_1,\ldots,\gamma_n\}$ in a group $\Gamma$, omnipotence says that we may find a homomorphism $f$ from $\Gamma$ to a finite group in which we may control the orders $o(f(\gamma_i))$; in particular, we may ensure that the orders are all distinct (see Definition \ref{def: Omnipotence} for the exact statement). Wise observed that, if every hyperbolic group is residually finite, then every finite independent set of elements of infinite order in every hyperbolic group is omnipotent.

In the appendix, Sisto proves the crucial technical Theorem \ref{thm: HHG theorem}, which gives the analogous statement for independent sets of pseudo-Anosovs in finite-index subgroups of mapping class groups. The proof builds on recent developments in the theory of Dehn fillings of hierarchically hyperbolic groups due to Behrstock--Hagen--Martin--Sisto \cite{behrstock_combinatorial_2024}, Dahmani \cite{dahmani_normal_2018}  and Dahmani--Hagen--Sisto \cite{dahmani_dehn_2021}. Although hierarchically hyperbolic groups are not mentioned in the statement of Theorem \ref{thm: HHG theorem}, it seems that the hierarchically hyperbolic framework developed in \cite{behrstock_hierarchically_2017} \emph{et seq.}\ is crucial for implementing the proof.

With Theorem \ref{thm: HHG theorem} in hand, the remainder of the proof of Theorem \ref{thm: Main theorem} is a short algebraic argument. However, it admits a geometric interpretation, and we take the opportunity to outline the geometric version of the argument here.

Fix a filling curve $a$ (defined up to isotopy) on $\Sigma$ and consider $\Mod(\Sigma).a$, its orbit under the action of $\Mod(\Sigma)$. The stabiliser of a filling curve is finite and so, after passing to a torsion-free congruence subgroup, the stabiliser of $a$ may be assumed to be trivial. For a finite-sheeted normal covering space $\Lambda$ of $\Sigma$ and a closed curve $c$ on $\Sigma$,  let $\deg_\Lambda(c)$ denote the degree with which $c$ unwraps in $\Lambda$. Let $\Gamma$ be a finite-index subgroup of $\Mod(\Sigma)$.   By applying Theorem \ref{thm: HHG theorem} to the (pure) mapping class group of the surface $\Sigma_*$ obtained by adding a puncture to $\Sigma$, we obtain a finite-sheeted normal covering space $\Lambda$ of $\Sigma$ with the the property that, for any $c\in\Mod(\Sigma).a$, $\deg_\Lambda(c)=\deg_\Lambda(a)$ if and only if $c\in\Gamma.a$.

To complete the proof, we need to show that the Veech subgroup $\Mod^\Lambda(\Sigma)$ of mapping classes that lift to $\Lambda$ is contained in $\Gamma$. Indeed, fix any elevation $\hat{a}$ of $a$ to $\Lambda$, and let $\phi\in\Mod^\Lambda(\Sigma)$ lift to $\hat{\phi}\in\Mod(\Lambda)$. Now $\hat{\phi}(\hat{a})$ is certainly an elevation of $\phi(a)\in\Mod(\Sigma).a$. However, since $\hat{\phi}$ is a lift of $\phi$, $\hat{\phi}(\hat{a})$ unwraps $\phi(a)$ with the same degree as $\hat{a}$ unwraps $a$, so
\[
\deg_\Lambda(\phi(a))=\deg_\Lambda(a)\,.
\]
Thus, by the defining property of $\Lambda$, $\phi(a)\in\Gamma.a$. Since the stabiliser $a$ was assumed to be trivial, it follows that $\phi\in\Gamma$, as required.

\subsection*{Outline}

The paper is structured as follows. In \S\ref{sec: CSP preliminaries} we record basic facts about the congruence subgroup property, including its interpretation in terms of Veech subgroups. Section \ref{sec: Omnipotence} develops Wise's notion of omnipotence, including stating the crucial Theorem \ref{thm: HHG theorem}.  Section \ref{sec: Level-k subgroups} collects some standard facts about the level-$k$ congruence subgroups of mapping class groups. In particular, the level-3 subgroup provides a convenient torsion-free congruence subgroup. Section \ref{sec: Proof of the main theorem} contains the proof of Theorem \ref{thm: Main theorem}; this is the only genuine novelty of the body of the paper. Section \ref{sec: Profinite rigidity} contains the proof of Theorem \ref{thm: Profinite rigidity}. It follows the outline of the main theorem of \cite{bridson_profinite_2017}, using recent developments of Liu \cite{liu_finite-volume_2023} and Cheetham-West \cite{cheetham-west_profinite_2024}.

Finally, in the appendix, Sisto proves Theorem \ref{thm: HHG theorem}, building on his recent work with Behrstock, Hagen and Martin \cite{behrstock_combinatorial_2024}.

\subsection*{Acknowledgements}

Thanks to Ian Agol, Marco Boggi, Tamunonye Cheetham-West and Dan Margalit for comments on an earlier draft.

\section{Preliminaries on the congruence subgroup property}\label{sec: CSP preliminaries}

In this section we collect useful basic facts about the congruence subgroup property. No originality is claimed: the material is all standard. Throughout, $G$ is a finitely generated group.

\begin{definition}[Congruence subgroups]
Consider a subgroup $\Delta\leq\Out(G)$. Let $Q$ be a finite, characteristic quotient of $G$. The kernel of the natural map
\[
\Delta\into\Out(G)\to\Out(Q)
\]
is called a \emph{principal congruence subgroup} of $\Delta$. A subgroup of $\Delta$ that contains a principal congruence subgroup is called a \emph{congruence subgroup}, and a quotient of $\Delta$ by a normal congruence subgroup is called a \emph{congruence quotient}. If every subgroup of finite index in $\Delta$ is congruence, then $\Delta$ is said to have the \emph{congruence subgroup property}.
\end{definition}

Note the slight abuse of terminology: whether or not $\Delta$ enjoys the congruence subgroup property depends on $G$, as the next example shows.

\begin{example}[$GL_2(\Z)$]\label{eg: GL(2,Z)}
Recall that
\[
\Out(\Z^2)\cong GL_2(\Z)\cong\Out(F_2)\,,
\]
where $F_2$ is the free group of rank 2. Asada proved that $\Out(F_2)$ has the congruence subgroup property \cite{asada_faithfulness_2001}. On the other hand, it is a classical fact, known to Fricke and Klein, that $\Out(\Z^2)$ does not have the congruence subgroup property \cite{raghunathan_congruence_2004}. We will outline a quick proof using the results of Stebe.

For each $n>0$, let $\eta_n:GL_2(\Z)\to GL_2(\Z/n\Z)$ be the natural homomorphism, and note that these form a cofinal family among the congruence quotients of $\Out(\Z^2)$. Stebe exhibited non-conjugate matrices $A,B\in GL_2(\Z)$ such that $\eta_n(A)$ and $\eta_n(B)$ are conjugate for every $n$ \cite[Example 1]{stebe_conjugacy_1972b}\footnote{The existence of such pairs is a standard consequence of class field theory.}.  In the same paper, Stebe showed that $GL_2(\Z)$ is conjugacy separable \cite[Theorem 1]{stebe_conjugacy_1972b}, meaning that there is some homomorphism $q$ to a finite group in which $q(A)$ and $q(B)$ are not conjugate. Therefore, $q$ does not factor through any $\eta_n$, so $\Out(\Z^2)$ does not have the congruence subgroup property.

This stands in contrast to the higher-rank case: Bass--Milnor--Serre  proved that $GL_k(\mathbb{Z})\cong\Out(\Z^n)$ does have the congruence subgroup property when $k\geq 3$ \cite{bass_solution_1967}.  By a similar argument, it follows that $GL_k(\mathbb{Z})$ is \emph{not} conjugacy separable for $k\geq 3$.
\end{example}

It is useful to note that congruence subgroups are closed under passing to intersections.

\begin{remark}\label{rem: Intersection of congruence subgroups}
For $K_1$ and $K_2$ both characteristic subgroups of finite index in a finitely generated group $G$, let $L$ be the intersection of all subgroups of index $|G:K_1\cap K_2|$ in $G$. Then $L$ is characteristic in $K_1\cap K_2$ and both natural homomorphisms $\Out(G)\to \Out(G/K_i)$ factor through $\Out(G)\to \Out(G/L)$. Therefore, a finite intersection of congruence subgroups is itself congruence.
\end{remark}

Rather than working directly with principal congruence subgroups, we will approach the congruence subgroup property from a slightly different point of view via \emph{Veech subgroups}, following \cite{ellenberg_arithmetic_2012}.

\begin{definition}[Veech subgroup]\label{def: Veech subgroup}
If $H$ is a normal subgroup of finite index in $G$ then
\[
\Delta_H:=\{\phi\in\Delta \mid \phi(H)=H\}
\]
is called a \emph{Veech subgroup}. (Note that, because $H$ is normal in $G$, $H$ is preserved by all inner automorphisms, so the assertion that $\phi(H)=H$ is independent of the choice of automorphism representing $\phi$.)
\end{definition}

Every Veech subgroup is a congruence subgroup, and hence one can prove the congruence subgroup property by showing that every subgroup of finite index contains a Veech subgroup.

\begin{lemma}\label{lem: Veech subgroups are congruence}
Every Veech subgroup is a congruence subgroup.
\end{lemma}
\begin{proof}
Let $H$ be a normal subgroup of finite index in $G$, and consider the corresponding Veech subgroup. Let $K$ be a characteristic subgroup of $G$ contained in $H$ (such as the intersection of all subgroups of index $|G:H|$).   Suppose that $\phi\in\Out(G)$ is in the principal congruence subgroup corresponding to $K$. By definition, any representative automorphism for $\phi$ descends to an inner automorphism $\bar{\phi}$ of $G/K$ and, in particular, $\bar{\phi}(H/K)=H/K$, so $\phi(H)=H$. This proves that the principal congruence subgroup corresponding to $K$ is contained in the Veech subgroup of $H$, as required.
\end{proof}

Thus, the congruence subgroup property follows if every subgroup of finite index contains a Veech subgroup.\footnote{In the case of free and surface groups, Lemma \ref{lem: Veech subgroups are congruence} has a converse: every congruence subgroup contains a finite intersection of Veech subgroups \cite{boggi_personal_2024}.} It is also useful to notice that elements of Veech subgroups descend to outer automorphism groups of quotients.

\begin{remark}\label{rem: Veech subgroups descend}
Let $H$ be a normal subgroup of finite index in $G$. Then $\phi\in\Out(G)$ descends to an outer automorphism $\bar{\phi}$ of $G/H$ if and only if $\phi$ is contained in the Veech subgroup $\Delta_H$.
\end{remark}

In this paper, we will be concerned with the setting where $\Sigma$ is a connected, orientable, hyperbolic surface of finite type, and $\Delta$ is the mapping class group $\Mod(\Sigma)$. A standard generalisation of the Dehn--Nielsen--Baer theorem asserts that the natural map $\Mod(\Sigma)\to\Out(\pi_1(\Sigma))$ is injective \cite[Theorem 8.8]{farb_primer_2012}, so the congruence subgroup property for $\Mod(\Sigma)$ makes sense.

We will write $\Mod^\Lambda(\Sigma)$ for the Veech subgroup corresponding to a finite-sheeted normal covering space $\Lambda\to\Sigma$; it consists of the mapping classes in $\Sigma$ that lift to $\Lambda$. Moreover, it will be technically convenient to replace $\Mod(\Sigma)$ by $\Mod_0(\Sigma)$, a torsion-free subgroup of finite index in the pure mapping class group of $\Sigma$, and we will set
\[
\Mod_0^\Lambda(\Sigma):=\Mod_0(\Lambda)\cap\Mod^\Lambda(\Sigma)\,,
\]
the Veech subgroup of $\Mod_0(\Sigma)$ corresponding to $\Lambda$.

\section{Omnipotence}\label{sec: Omnipotence}

In order to prove the congruence subgroup property, given a finite-index subgroup $\Gamma$ of $\Mod(\Sigma)$, we need to construct a certain normal subgroup of finite index in $\pi_1(\Sigma)$. Our route to this will be via a property called \emph{omnipotence}, which makes sense for \emph{independent} subsets.

\begin{definition}[Independent sets]
A finite set $\{\gamma_i\}$ of non-trivial elements in a group $G$ is called \emph{independent} if, whenever there are non-zero integers $m$ and $n$ such that $\gamma_i^m$ is conjugate to $\gamma_j^n$, it follows that $i=j$.
\end{definition}

In practice, independent subsets usually arise from almost malnormal families of virtually cyclic subgroups.

\begin{definition}[Almost malnormal]\label{def: Malnormal family}
A family of subgroups $\{P_i\}$ of a group $G$ is called \emph{almost malnormal} if
\[
|gP_ig^{-1}\cap P_j|<\infty
\]
implies that $i=j$ and $g\in P_i$. If 
\[
|gP_ig^{-1}\cap P_j|=1
\]
implies that $i=j$ and $g\in P_i$ then $\{P_i\}$ is said to be \emph{malnormal}. (Note that the two notions coincide if $G$ is torsion-free.) A subgroup $P$ is said to be (almost) malnormal whenever the singleton family $\{P\}$ is.
\end{definition}

\begin{remark}\label{rem: Almost malnormality and independence}
Suppose that $\{P_i\}$ is an almost malnormal family of infinite subgroups and, for each $i$, $\langle\gamma_i\rangle$ has finite index in $P_i$. Then $\{\gamma_i\}$ is independent.
\end{remark}

A malnormal family naturally pulls back to a malnormal family in a normal subgroup of finite index.\footnote{With a little more care, pullbacks can also be defined for subgroups that are not normal, nor of finite index. However, we will not need that extra subtlety here.}

\begin{remark}\label{rem: Pullback of a malnormal family}
Let $\curlyP=\{P_i\mid i\in I\}$ be a malnormal family of subgroups of a group $G$ and let $H$ be a normal subgroup of finite index in $G$. For each $i$, let $n_i=|P_i:P_i\cap H|$, and fix a set of coset representatives $\{p_{ij}\mid 1\leq j\leq n_i\}$ for $P_i\cap H$ in $P_i$. Then
\[
\curlyP^H:=\{p_{ij}(P_i\cap H)p_{ij}^{-1}\mid i\in I\,,\,1\leq j\leq n_{i}\}
\]
is a malnormal family in $H$. Note furthermore that $\curlyP^H$ is well defined up to conjugation in $H$.
\end{remark}

It is well known that maximal cyclic subgroups in torsion-free hyperbolic groups (for instance, hyperbolic surface groups) are malnormal. We will be interested in sets of pseudo-Anosovs in subgroups of mapping class groups. In this case, almost malnormality is a consequence of Bowditch's theorem that the action on the curve complex is acylindrical \cite{bowditch_tight_2008}. The following result appears in a paper of Dahmani--Guirardel--Osin  \cite{dahmani_hyperbolically_2017}, although it is implicit in an earlier preprint of McCarthy \cite{mccarthy_normalizers_1982}.

\begin{proposition}\label{prop: Elementary closure of a pseudo-Anosov}
Let $\Sigma$ be an orientable, hyperbolic surface of finite type. Any pseudo-Anosov $\phi\in\Mod(\Sigma)$ is contained in a unique maximal virtually cyclic subgroup $E(\phi)$, and $E(\phi)$ is almost malnormal.
\end{proposition}
\begin{proof}
The existence of $E(\phi)$ is \cite[Theorem 2.19(a)]{dahmani_hyperbolically_2017}, while the almost malnormality follows from \cite[Proposition 2.10]{dahmani_hyperbolically_2017}.
\end{proof}

The maximal virtually cyclic subgroup $E(\phi)$ is called the \emph{elementary closure} of $\phi$.

Omnipotence was introduced by Wise \cite{wise_subgroup_2000}, who framed it as a property of a group $G$. However, we will think of it as a property of an independent subset.

\begin{definition}[Omnipotence]\label{def: Omnipotence}
An independent subset $\{\gamma_i\}$ of a group $G$ is called \emph{omnipotent} if the following holds: there is a positive integer $K\in \Z_{>0}$ such that, for every choice of positive integers $\{n_i\}$, there is a homomorphism $f$ from $G$ to a finite group such that the order of $f(\gamma_i)$ equals $Kn_i$.
\end{definition}

Wise observed that, if every hyperbolic group is residually finite, then every independent finite set of infinite-order elements in every hyperbolic group is omnipotent.  He then confirmed this directly for free groups \cite[Theorem 3.5]{wise_subgroup_2000}.

\begin{theorem}[Wise]\label{thm: Omnipotence of free groups}
Any finite, independent subset of a free group is omnipotent.
\end{theorem}

The corresponding result for surface groups was proved by Bajpai \cite{bajpai_omnipotence_2007} (see also \cite{wilton_virtual_2010}). Via a much deeper argument using the machinery of special cube complexes, Wise extended Theorem \ref{thm: Omnipotence of free groups} to all virtually special hyperbolic groups \cite[Theorem 16.6]{wise_structure_2021}. The key ingredient in Wise's proof is his malnormal special quotient theorem \cite[Theorem 12.2]{wise_structure_2021}, which enables one to Dehn fill while preserving residual finiteness.

Recent advances in our understanding of Dehn fillings of hierarchically hyperbolic groups \cite{behrstock_combinatorial_2024,dahmani_normal_2018,dahmani_dehn_2021} make it possible to prove a similar result for pseudo-Anosovs in finite-index subgroups of mapping class groups, assuming that every hyperbolic group is residually finite. The following theorem is proved by Sisto in Appendix \ref{app: Omnipotence via HHGs}.

\begin{theorem}[Sisto]\label{thm: HHG theorem}
Let $\Sigma$ be a connected, orientable, hyperbolic surface of finite type and let $\Gamma$ be a subgroup of finite index in $\Mod(\Sigma)$. Suppose that every hyperbolic group is residually finite.  Then every finite, independent set of pseudo-Anosovs $\{\phi_i\}$ in $\Gamma$ is omnipotent.
\end{theorem}

\begin{remark}\label{rem: Omnipotence and finite-index}
Omnipotence does not naively pass to subgroups of finite index.  Therefore, Theorem \ref{thm: HHG theorem} cannot be replaced with the simpler statement that finite independent sets of pseudo-Anosovs in mapping class groups are omnipotent.
\end{remark}

Theorem \ref{thm: HHG theorem} is the only place that the assumption that hyperbolic groups are residually finite is used in the proof of Theorem \ref{thm: Main theorem}.  Therefore, an unconditional proof of Theorem \ref{thm: HHG theorem} would give an unconditional proof of the congruence subgroup property for mapping class groups.

At this point, it is natural to hope that mapping class groups could be virtually special. If so, then perhaps a version of Wise's techniques could be applied to prove Theorem \ref{thm: HHG theorem} unconditionally. Unfortunately, this approach fails: Kapovich--Leeb showed that mapping class groups of surfaces of genus at least 3 cannot act properly on CAT(0) spaces \cite[Theorem 4.2]{kapovich_actions_1996}.

\section{Level-\texorpdfstring{$k$}{k} subgroups}\label{sec: Level-k subgroups}

In this section we collect some standard results about explicit subgroups of finite index in mapping class groups. In particular, we will see that the level-3 subgroup provides a convenient torsion-free congruence subgroup of the pure mapping class group.

Let $\Sigma$ be an orientable surface of finite type. For any integer $k$, the natural action of $\Mod(\Sigma)$ on $H_1(\Sigma,\Z/k\Z)$ defines a homomorphism from $\Mod(\Sigma)$ to the finite linear group $GL(b_1(\Sigma),k)$. The \emph{level-$k$} mapping class group $\Mod(\Sigma)[k]$ is defined to be the kernel of this homomorphism. Since $H_1(\Sigma,\Z/k\Z)$ is a characteristic quotient of $\pi_1(\Sigma)$, $\Mod(\Sigma)[k]$ is congruence. Furthermore, Serre observed that the level-$k$ subgroups are torsion free, unless $k=1,2$ \cite[Theorem 6.9]{farb_primer_2012}.

\begin{proposition} \label{prop: Torsion-free subgroups}
For $k\geq 3$, $\Mod(\Sigma)[k]$ is torsion-free.
\end{proposition}

The level-$k$ subgroups provide explicit torsion-free congruence subgroups, but the existence of such subgroups can also be deduced from much more general facts. Indeed, Grossman proved that $\Mod(\Sigma)$ is residually finite, and the quotients that her argument produces are congruence by construction \cite{grossman_residual_1974}. Since mapping class groups have finitely many conjugacy classes of torsion elements \cite[Theorem 7.14]{farb_primer_2012}, it follows that there is a torsion-free congruence subgroup.

We next turn to another useful subgroup of finite index. The action of $\Mod(\Sigma)$ on the $n$ punctures of $\Sigma$ defines a natural homomorphism to the symmetric group $\Mod(\Sigma)\to S_n$; the kernel is the \emph{pure mapping class group} $\PMod(\Sigma)$. The author is grateful to Dan Margalit for drawing his attention to the next lemma, which explains the relationship between level-$k$ subgroups and pure mapping class groups.

\begin{lemma}\label{lem: Level subgroups and pure subgroups}
Let $\Sigma$ be an orientable surface of finite type. If $k\geq 2$ then $\Mod(\Sigma)[k]\leq\PMod(\Sigma)$.
\end{lemma}
\begin{proof}
A collection of simple closed curves around the $n$ punctures span an $(n-1)$-dimensional summand $M$ of $H_1(\Sigma,\Z/k\Z)$. Note that $M$ is naturally a $\Mod(\Sigma)$-submodule, and the $\Mod(\Sigma)$ action descends to an $S_n$-action, making $M$ into an $S_n$-module. By direct calculation, the regular $S_n$-module over $\Z/k\Z$ is the direct sum of $M$ with the trivial module. It follows that $M$ is a faithful $S_n$-module. Thus, any mapping class that acts trivially on $H_1(\Sigma,\Z/k\Z)$ also acts trivially on the punctures, which implies the result.
\end{proof}

It follows from Lemma \ref{lem: Level subgroups and pure subgroups} that $\PMod(\Sigma)$ is congruence. We finish the section with an alternative proof of this fact, using Theorem \ref{thm: Omnipotence of free groups}. Although not needed for the other reuslts of the paper, the proof illustrates how omnipotence can be used to construct congruence subgroups.

\begin{proposition}\label{prop: Pure mapping class groups are congruence}
Let $\Sigma$ be an orientable, hyperbolic surface of finite type. Then $\PMod(\Sigma)$ is congruence in $\Mod(\Sigma)$.
\end{proposition}
\begin{proof}
Let $\{x_1,\ldots,x_n\}$ denote the punctures of $\Sigma$ and, for each $i$, let  $\alpha_i$ be a based loop winding once around  $x_i$, representing an element of $\pi_1(\Sigma)$.  Because $\Sigma$ is hyperbolic, the family
\[
\{ \alpha_1,\ldots,\alpha_n\}
\]
is independent. Therefore, by Theorem \ref{thm: Omnipotence of free groups}, there is a homomorphim $f$ from $\pi_1(\Sigma)$ to a finite group $Q$ such that $o(f(\alpha_i))=Ki$ for each $i$.\footnote{In fact, we just need that the orders $o(f(\alpha_i))$ are pairwise distinct. In this particular case, rather than appealing to Theorem \ref{thm: Omnipotence of free groups}, it is not difficult to achieve this by a direct construction.}  Let $\Lambda$ be the normal covering space of $\Sigma$ corresponding to the kernel of $f$. It remains to prove that the Veech subgroup  $\Mod^\Lambda(\Sigma)$ is contained in $\PMod(\Sigma)$.

Consider $\phi\in\Mod^\Lambda(\Sigma)$. This means that any choice of representative automorphism $\phi_*\in\Aut(\pi_1(\Sigma))$ preserves $\pi_1(\Lambda)$, and so descends to an automorphism $\bar{\phi}_*$ of $Q$.

Suppose that $\phi(x_i)=x_j$. Then the loop $\phi\circ\alpha_i$ is freely homotopic to $\alpha_j$, so $f(\alpha_j)$ is conjugate to $f(\phi_*(\alpha_i))=\bar{\phi}_*(f(\alpha_i))$. Since $\bar{\phi}_*$ is an automorphism of $Q$, it follows that
\[
Kj=o(f(\alpha_j))=o(f(\alpha_i))=Ki\,,
\]
whence $i=j$. That is, $\phi(x_i)=x_i$ for all $i$, so $\phi\in\PMod(\Sigma)$ as required.
\end{proof}

\section{Proof of the congruence subgroup property}\label{sec: Proof of the main theorem}

As well as Theorem \ref{thm: HHG theorem}, the proof of Theorem \ref{thm: Main theorem} relies on the Birman exact sequence \cite{birman_mapping_1969} (see also \cite[Theorem 4.6]{farb_primer_2012}). Let $\Sigma$ be an orientable, hyperbolic surface of finite type and let $\Sigma_*$ be the result of adding a single puncture. The Birman exact sequence is
\[
1\to\pi_1(\Sigma)\stackrel{\iota}{\to} \PMod(\Sigma_*)\to\PMod(\Sigma)\to 1
\] 
where $\iota$ is the \emph{point-pushing map} and the \emph{forgetful map} $\PMod(\Sigma_*)\to\PMod(\Sigma)$ forgets the puncture. Algebraically, these maps can be understood via the assertion that the natural diagram
\begin{center}
  \begin{tikzcd}
1\arrow{r}&\pi_1(\Sigma)\arrow{r}{\iota}\arrow{d}{=} &\PMod(\Sigma_*)\arrow{r}\arrow{d} &  \PMod(\Sigma)\arrow{d}\arrow{r}&1 \\
1\arrow{r}&\pi_1(\Sigma)\arrow{r} & \Aut(\pi_1\Sigma) \arrow{r}& \Out(\pi_1\Sigma)\arrow{r}&1 
  \end{tikzcd}
\end{center}
commutes (cf.\ \cite[p.\ 235]{farb_primer_2012}). When $\Sigma$ is hyperbolic, the vertical arrows are injective. For a mapping class $\phi\in\PMod(\Sigma)$, a choice of preimage will be denoted by $\phi_*$. Note that $\phi_*$ is naturally an automorphism of $\pi_1(\Sigma)$, while $\phi$ is the corresponding \emph{outer} automorphism.\footnote{On a few occasions, when we need to consider sets in $\Mod(\Sigma_*)$ indexed by subscripts, we will abandon this convention and just use $\phi_i$, say, to denote an element of $\Mod(\Sigma_*)$. However, on these occasions we will make no reference to the corresponding outer automorphism, so no confusion should be caused.}

\begin{remark}\label{rem: Action on point-pushing subgroup}
The commutativity of the above diagram enables us to compute the action of elements of $\phi_*\in\PMod(\Sigma_*)$ on the point-pushing subgroup $\iota(\pi_1(\Sigma))$. Specifically,
\[
\iota(\phi_*(g))=\phi_*\iota(g)\phi_*^{-1}
\]
for any $g\in \pi_1(\Sigma)$ and any $\phi_*\in\PMod(\Sigma_*)\leq\Aut(\pi_1(\Sigma))$.
\end{remark}

The next theorem gives a conditional proof of the congruence subgroup property.

\begin{theorem}\label{thm: Main result in the torsion-free case}
Let $\Sigma$ be a connected, orientable, hyperbolic surface of finite type, and let $\Sigma_*$ be the surface obtained by adding one puncture. Suppose that the following is true: for every subgroup $\Gamma_*$ of finite index in $\Mod(\Sigma_*)$, every finite, independent set of pseudo-Anosovs $\{\phi_i\}$  in $\Gamma_*$ is omnipotent. Then $\Mod(\Sigma)$ has the congruence subgroup property.
\end{theorem}
\begin{proof}
Let $\Mod_0(\Sigma)=\Mod(\Sigma)[3]$, which is contained in $\PMod(\Sigma)$ by Lemma \ref{lem: Level subgroups and pure subgroups}. By Remark \ref{rem: Intersection of congruence subgroups}, it suffices to prove that  $\Mod_0(\Sigma)$ has the congruence subgroup property.

Consider $\Gamma$ a subgroup of finite index in $\Mod_0(\Sigma)$. By replacing $\Gamma$ with the intersection of its conjugates, we may assume that $\Gamma$ is normal. Let $\Mod_0(\Sigma_*)$ be the preimage of $\Mod_0(\Sigma)$ under the forgetful map, and let $\Gamma_*$ be the preimage of $\Gamma$.  Let $\{1=\phi_1,\phi_2,\ldots,\phi_k\}$ be a set of coset representatives for $\Gamma_*$ in $\Mod_0(\Sigma_*)$, where $k=|\Mod_0(\Sigma_*):\Gamma_*|=|\Mod_0(\Sigma):\Gamma|$.

Fix a filling curve $a$ on $\Sigma$ that does not represent a proper power in $\pi_1(\Sigma)$.  Without loss of generality, $a$ is based at the added puncture of $\Sigma_*$ and so defines an element of $\pi_1(\Sigma)$. By Kra's theorem \cite{kra_nielsen-thurston-bers_1981}, the image $\alpha_*=i(a)$ under the point-pushing map is a pseudo-Anosov element of $\Mod_0(\Sigma_*)$.

We claim that $\langle \alpha_*\rangle$ is malnormal in $\Mod_0(\Sigma_*)$.  To this end, consider $\phi_*\in \Mod_0(\Sigma_*)$ such that $\phi_* \alpha_*^m\phi_*^{-1}=\alpha_*^n$, with $m,n\neq 0$.   Consider the elementary closure $E(\alpha_*)$, and let $E_0(\alpha_*)=E(\alpha_*)\cap \Mod_0(\Sigma_*)$. By Proposition \ref{prop: Elementary closure of a pseudo-Anosov}, $E(\alpha_*)$ is almost malnormal in $\Mod(\Sigma_*)$, so $\phi_*\in E_0(\alpha_*)$.  Since $\alpha_*$ is in the point-pushing subgroup, the image of $E(\alpha_*)$ in $\Mod(\Sigma)$ is finite. Hence, the image of $E_0(\alpha_*)$  is trivial, because $\Mod_0(\Sigma)$ is torsion-free by Proposition \ref{prop: Torsion-free subgroups}.  Therefore, $\phi_*$ is in the image of the point-pushing map, and so $\phi_*\in\langle\alpha_*\rangle$ as required, by the malnormality of maximal cyclic subgroups of $\pi_1(\Sigma)$.

Write $\alpha_i=\phi_i\alpha_*\phi_i^{-1}$ for $1\leq i\leq k$. By Remark \ref{rem: Pullback of a malnormal family}, $\{\langle \alpha_i\rangle\mid 1\leq i\leq k\}$ is a malnormal family of cyclic subgroups in $\Gamma_*$. In particular, $\{\alpha_i\mid 1\leq i\leq k\}$ is an independent set of pseudo-Anosovs in $\Gamma_*$, so is omnipotent by assumption. Let $K$ be the corresponding constant.

Omnipotence provides a homomorphism to a finite group $f:\Gamma_*\to Q$ such that the order $o(f(\alpha_1))=K$ while $o(f(\alpha_i))=2K$ for $i>1$. Let $\Lambda$ be the finite-sheeted normal covering space of $\Sigma$ corresponding to the kernel of $f\circ\iota$. It remains to prove that $\Mod^\Lambda_0(\Sigma)\subseteq\Gamma$.

To this end, consider $\psi\in\Mod^\Lambda_0(\Sigma)$ with a choice of preimage $\psi_*$ under the forgetful map. At the level of fundamental groups, this means that  $\psi_*$ preserves $\pi_1(\Lambda)$, and so descends to an automorphism $\bar{\psi}_*$ of the finite group $\pi_1(\Sigma)/\pi_1(\Lambda)\cong f\circ\iota(\pi_1(\Sigma))\leq Q$. Therefore,
\begin{equation}\label{eqn: Order of alpha1} 
o(f\circ \iota(\psi_*(a)))=o(\bar{\psi}_*(f\circ\iota(a)))=o(f(\alpha_1))=K\,.
\end{equation}
On the other hand, because the $\phi_i$ are coset representatives for $\Gamma_*$, $\psi_*=\gamma_*\phi_i$ for some $i$ and some $\gamma_*\in\Gamma_*$. Therefore,
\[
\iota(\psi_*(a))=\psi_*\alpha_*\psi_*^{-1}=\gamma_*\phi_i\alpha_*\phi_i^{-1}\gamma_*^{-1}=\gamma_*\alpha_i\gamma_*^{-1}
\]
by Remark \ref{rem: Action on point-pushing subgroup}, so
\begin{equation}\label{eqn: Order of alphai} 
o(f\circ\iota(\psi_*(a)))=o(f(\gamma_*\alpha_i\gamma_*^{-1}))=o(f(\alpha_i))\,.
\end{equation}
Putting (\ref{eqn: Order of alpha1}) and (\ref{eqn: Order of alphai}) together, we get $o(f(\alpha_i))=K$ which implies, by the choice of $f$, that $i=1$. Since $\phi_1=1$ it follows that $\psi_*=\gamma_*\in\Gamma_*$, so $\psi\in\Gamma$. This proves that $\Mod^\Lambda_0(\Sigma)\subseteq\Gamma$.

In summary, $\Gamma$ contains a Veech subgroup, so $\Gamma$ is a congruence subgroup by Lemma \ref{lem: Veech subgroups are congruence}, as required.
\end{proof}

Theorem \ref{thm: Main theorem} now follows by combining Theorems \ref{thm: HHG theorem} and \ref{thm: Main result in the torsion-free case}

\section{Profinite rigidity}\label{sec: Profinite rigidity}

The purpose of this section is to prove Theorem \ref{thm: Profinite rigidity}, about profinite rigidity of finite-volume Kleinian groups, again assuming that every hyperbolic group is residually finite.

The proof is by now fairly standard. It is an elaboration of the author's proof with Bridson and Reid of profinite rigidity for punctured torus bundles over the circle \cite{bridson_profinite_2017}, and is close to the recent work of Cheetham-West \cite{cheetham-west_profinite_2024}. Nevertheless, since the exact statement could be of use in proving the existence of a non-residually finite group, we give some details here.

The starting point is a deep theorem of Liu \cite{liu_finite-volume_2023}, which in turn builds on earlier work of Liu \cite{liu_mapping_2023} and Jaikin-Zapirain \cite{jaikin-zapirain_recognition_2020}.  Given two hyperbolic 3-manifolds $M$ and $N$ that are profinitely equivalent in the sense that $\widehat{\pi_1(M)}\cong\widehat{\pi_1(N)}$, Liu explains how to translate a fibration of $M$ to a corresponding fibration of $N$ \cite[Theorem 6.1, Corollary 6.2]{liu_finite-volume_2023}.

In a recent preprint, Cheetham-West shows that, under Liu's correspondence, the profinite completion determines the topological type of the fibre surface. The following statement follows by combining \cite[Theorem 6.1, Corollary 6.2]{liu_finite-volume_2023} with \cite[Proposition 4.1]{cheetham-west_profinite_2024} (see also \cite[Theorem 1.4]{cheetham-west_profinite_2024}).

\begin{theorem}[Liu, Cheetham-West]\label{thm: Liu's theorem}
Let $\phi$ be a pseudo-Anosov mapping class of an orientable hyperbolic $\Sigma$ of finite type. Let $M_\phi$ be the mapping torus, and suppose that $\widehat{\pi_1(M_\phi)}\cong \widehat{\Gamma}$ for some Kleinian group $\Gamma$. Then $\Gamma\cong\pi_1(M_\psi)$ for some pseudo-Anosov $\psi\in\Mod(\Sigma)$. Furthermore, the images of $\phi$ and $\psi$ have equal orders in every congruence quotient of $\Mod(\Sigma)$.
\end{theorem}

Assuming the residual finiteness of hyperbolic groups, the next proposition complements the conclusions of Theorem \ref{thm: Liu's theorem}.

\begin{proposition}\label{prop: Orders of pseudo-Anosovs}
Suppose that all hyperbolic groups are residually finite. Let $\Sigma$ be an orientable hyperbolic surface of finite type and let $\phi,\psi$ be pseudo-Anosovs on $\Sigma$. Either there is a positive integer $n$ such that $\phi^n$ is conjugate to $\psi^{\pm n}$ in $\Mod(\Sigma)$, or there is a congruence quotient of $\Mod(\Sigma)$ in which the images of $\phi$ and $\psi$ have different orders.
\end{proposition}
\begin{proof}
By Theorem \ref{thm: Main theorem}, if all hyperbolic groups are residually finite then every finite quotient of $\Mod(\Sigma)$ is congruence. Therefore, it suffices to exhibit any finite quotient of $\Mod(\Sigma)$ in which the images of $\phi$ and $\psi$ have different orders.

Suppose first that the set $\{\phi,\psi\}$ is independent in $\Mod(\Sigma)$. By Theorem \ref{thm: HHG theorem}, there is positive integer $K$ and a homomorphism $\eta$ from $\Mod(\Sigma)$ to a finite group such that $o(\eta(\phi))=K$ but  $o(\eta(\psi))=2K$, and the result follows.

Otherwise, $\phi^m$ is conjugate to $\psi^{\pm n}$ for some positive integers $m,n$. Applying Theorem \ref{thm: HHG theorem} to the singleton $\{\phi\}$, there is a homomorphism $\eta$ from $\Mod(\Sigma)$ to a finite group such that $o(\eta(\phi))=Kmn$, for some positive integer $K$. In particular, $o(\eta(\phi))=Kn$ while $o(\eta(\phi))=Km$, so the result follows unless $m=n$, as required.
\end{proof}

At this point, we can see that profinitely equivalent mapping tori must be cyclically commensurable. If the monodromy is \emph{asymmetric} then cyclic commensurability can be upgraded to homeomorphism.

\begin{definition}[Asymmetric pseudo-Anosovs]\label{def: Asymmetric pseudo-Anosov}
A pseudo-Anosov $\phi$ on an orientable hyperbolic surface $\Sigma$ of finite type is \emph{symmetric} if there is proper finite-sheeted covering of orbifolds $\Sigma\to\curlyO$ such that $\phi$ descends to $\curlyO$; otherwise, $\phi$ is called \emph{asymmetric}. If there is such a normal covering (i.e.\ if $\curlyO$ is the quotient of $\Sigma$ by a finite group of isometries), then $\phi$ is called \emph{strongly symmetric}; otherwise, $\phi$ is called \emph{weakly asymmetric}.
\end{definition}

Weak asymmetry has a convenient algebraic characterisation in terms of the elementary closure.

\begin{remark}\label{rem: Algebraic characterisation of being weakly asymmetric}
Let $\Sigma$ be an orientable hyperbolic surface of finite type. A pseudo-Anosov $\phi\in\Mod(\Sigma)$ is weakly asymmetric if and only if the maximal virtually cyclic subgroup that contains it, $E(\phi)$, is either cyclic or dihedral. Indeed, $E(\phi)$ has a unique maximal finite subgroup $K$, and $E(\phi)/K$ is either cyclic or dihedral. By Nielsen realisation, $K$ can be realised by isometries of $\Sigma$ \cite{kerckhoff_nielsen_1983}. If $K$ is non-trivial then $\phi$ descends to the quotient orbifold $\curlyO=K\backslash \Sigma$, so is strongly symmetric. Conversely, if $\phi$ descends to a proper regular quotient orbifold $\curlyO=K\backslash \Sigma$ then $\phi$ normalises the natural copy of $K$ inside $\Mod(\Sigma)$, so $K\langle\phi\rangle\leq E(\phi)$, and $E(\phi)$ is not cyclic or dihedral.
\end{remark}

We are now ready to prove Theorem \ref{thm: Profinite rigidity}.

\begin{proof}[Proof of Theorem \ref{thm: Profinite rigidity}]
We first prove part (i). By Agol's virtual fibring theorem \cite{agol_virtual_2013}, after passing to finite-index subgroups of $G$ and $H$, we may assume that $G$ is a the fundamental group of a mapping torus $M_\phi$, for some pseudo-Anosov mapping class $\phi$ of a surface $\Sigma$ of finite type. By Theorem \ref{thm: Liu's theorem}, $H$ is the fundamental group of a mapping torus $M_\psi$ for some pseudo-Anosov $\psi$ on $\Sigma$, and furthermore the images of $\phi$ and $\psi$ have the same order in every congruence quotient. Assuming every hyperbolic group is residually finite, Proposition \ref{prop: Orders of pseudo-Anosovs} implies that $\phi$ and $\psi$ have conjugate powers. Hence $M_\phi$ and $M_\psi$ are (cyclically) commensurable, as required.

For the second part, suppose that $G=\pi_1(M_\phi)$ for weakly asymmetric $\phi$. As above, $H=\pi_1(M_\psi)$ with $\psi\in E(\phi)$, and $\phi^n=\psi^{\pm n}$ for some positive $n$. Because $\phi$ is weakly asymmetric, $E(\phi)$ is infinite cyclic of infinite dihedral. But infinite-order elements of $\Z$ or $D_\infty$ have unique roots, so $\phi=\psi^{\pm 1}$, and the result follows.
\end{proof}

\appendix

\section{Omnipotence in finite-index subgroups of mapping class groups by Alessandro Sisto}\label{app: Omnipotence via HHGs}

The purpose of this appendix is to prove Theorem \ref{thm: HHG theorem}. We will use results from \cite{behrstock_combinatorial_2024}, which we summarise in the following.

\begin{theorem}
\label{thm:extract}
 Let $\Sigma$ be an orientable, hyperbolic surface of finite type and let $\Gamma$ be a subgroup of finite index in $\Mod(\Sigma)$. Let $Q$ be a convex-cocompact subgroup of $G=\Mod(\Sigma)$. If all hyperbolic groups are residually finite then there exists a quotient $\pi:G\to\bar G$ with the following properties:
 \begin{enumerate}
   \item\label{item:in_Gamma} The kernel of $\phi$ is contained in $\Gamma$.
 \item\label{item:hyp} $\bar G$ is hyperbolic.
    \item\label{item:inj} $\phi|_Q$ is injective.

  For the next items, let $X$ be the curve graph of $\Sigma$ and let $\bar X$ be its quotient by the kernel of $\phi$, with quotient map $\psi:X\to\bar X$; note that $\bar G$ naturally acts on $\bar X$. Fix a basepoint $x_0\in X$, with image $\bar x_0$ in $\bar X$.

  \item\label{item:cobound_isom} For all $q\in Q$ we have that $\psi|_{[x_0,qx_0]}$ is an isometric embedding.

  \item\label{item:quad_lift} For every geodesic quadrangle $\bar\gamma_1\ast\dots\ast \bar\gamma_4$ in $\bar X$ there exists a geodesic quadrangle $\gamma_1\ast\dots\ast \gamma_4$ in $X$ with $\psi(\gamma_i)=\bar\gamma_i$.

  \item\label{item:lift_cobound} Moreover, in the setting of the previous item, if a $\bar\gamma_i$ is a translate of a geodesic of the form $\psi([x_0,qx_0])$ for some $q\in Q$, then the corresponding $\gamma_i$ is a translate of $[x_0,qx_0]$.
 \end{enumerate}

\end{theorem}

\begin{proof}
 The main output of the construction of \cite{behrstock_combinatorial_2024} is given in \cite[Theorem 7.1]{behrstock_combinatorial_2024}, with additional properties stated in \cite[Proposition 8.13]{behrstock_combinatorial_2024}. Note in particular that \cite[Proposition 8.13-(IV)]{behrstock_combinatorial_2024} stipulates that all kernels of the quotient maps described by the statement are contained in a fixed finite-index subgroup of $G$, coming from \cite[Lemma 8.1]{behrstock_combinatorial_2024}. For our purposes, we can and will further assume that said subgroup is contained in $\Gamma$, which can be arranged simply by intersecting the subgroup given by \cite[Lemma 8.1]{behrstock_combinatorial_2024} with $\Gamma$. This guarantees that item \ref{item:in_Gamma} holds.

 We now have to choose specific input for the construction, the required data being a finite subset $F$ of $G$, an integer $i$, and a convex-cocompact subgroup $Q$. The finite subset plays no role for us and can be chosen arbitrarily, we set $i$ to be the complexity of $\Sigma$ minus 1, and we take $Q$ to be our given convex-cocompact subgroup.

 The quotient $\bar G=\bar G_i$ given by \cite[Theorem 7.1]{behrstock_combinatorial_2024} is hyperbolic (item \ref{item:hyp}) because $\bar G$ has an HHS structure where all hyperbolic spaces involved are bounded (with uniform diameter bound) except for one; this is \cite[Theorem 7.1(II)]{behrstock_combinatorial_2024}. By the distance formula for HHSs, $\bar G$ is quasi-isometric to said hyperbolic space, so that it is hyperbolic. (One can also use the same argument as in \cite[Corollary 7.7]{behrstock_combinatorial_2024}, which is slightly more complicated than necessary). Moreover, by \cite[Theorem 7.1-(III)]{behrstock_combinatorial_2024}, $Q$ embeds in $\bar G$, which is our item \ref{item:inj}.

 Item \ref{item:cobound_isom} is Lemma \cite[Lemma 8.44]{behrstock_combinatorial_2024}, while item \ref{item:quad_lift} is given by the first part of \cite[Proposition 8.13-(VI)]{behrstock_combinatorial_2024}. Finally, the ``moreover'' part of \cite[Proposition 8.13-(VI)]{behrstock_combinatorial_2024} gives item \ref{item:lift_cobound}.
 \end{proof}

\begin{proof}[Proof of Theorem \ref{thm: HHG theorem}]
 Applying Theorem \ref{thm:extract} with $Q$ any free convex-cocompact subgroup freely generated by conjugates of powers of the $\phi_i$, we obtain a hyperbolic quotient $\pi:G\to\bar G$ of $G=\Mod(\Sigma)$ with the property that all $\bar\phi_i=\pi(\phi_i)$ are infinite-order elements.

 \par\medskip

 {\bf Claim.} We have $E(\bar\phi_i)=\pi(E(\phi_i))$ for all $i$. Moreover, $\{\bar\phi_i\}$ is independent.


 \begin{proof}[Proof of Claim]
 We use the same notation for the curve graph $X$ and its quotient $\bar X$ as in Theorem \ref{thm:extract}. Note that the inclusion $\pi(E(\phi_i))<E(\bar\phi_i)$ is clear from the fact that $\pi(E(\phi_i))$ is virtually cyclic and contains $\bar\phi_i$.

Independence of $\{\phi_i\}$ together with the fact that the $\phi_i$ are WPD \cite{bestvina_bounded_2002} for the action on $X$ yield that for all $C\geq 0$ there exists $N\geq 0$ such that the following holds. If $\gamma\in \Gamma$ is such that for some integers $k,l$ with $|k|,|l|>N$ we have $d_X(x_0,\gamma x_0),d_X(\phi_i^kx_0,\gamma\phi_j^lx_0)\leq C$ for some $\phi_i$ and $\phi_j$, then $i=j$ and $\gamma\in E(\phi_i)$. (That is to say, orbits of distinct cosets of the $\phi_i$ cannot fellow-travel for an arbitrarily long distance, except for cosets within the same $E(\phi_i)$.)

 Suppose now that we have $\bar\gamma\in\bar\Gamma$ such that $\bar\gamma\bar\phi_i^k\bar\gamma^{-1}=\bar\phi_i^l$ for some $\bar\phi_i,\bar\phi_j$ and integers $k,l\neq 0$. We have to show $i=j$ and $\bar\gamma\in \pi(E(\phi_i))$. For the $N$ given by the property stated above with $C=d_{\bar X}(\bar x_0,g\bar x_0)$, we can assume $|k|,|l|>N$. We now have a geodesic quadrangle $\bar{\mathcal Q}$ in $\bar X$ with vertices $\bar x_0$, $\bar\gamma\bar x_0$, $\bar \gamma\bar\phi_j^l\bar x_0$, $\bar\phi_i^k\bar x_0$, which has two sides of length $C$. Moreover, by Theorem \ref{thm:extract}-\ref{item:cobound_isom} we can take the sides $[\bar x_0,\bar\phi_i^k\bar x_0]$ and $[\bar\gamma\bar x_0,\bar \gamma\bar\phi_j^l\bar x_0]$ to be images in $\bar X$ of translates of geodesics $[x_0,\phi_i^kx_0]$ and $[x_0,\phi_j^l x_0]$.

 According to Theorem \ref{thm:extract}-\ref{item:quad_lift}, this geodesic quadrangle can be lifted to $X$. By Theorem \ref{thm:extract}-\ref{item:lift_cobound}, it can in fact be lifted in a way that the vertices of the lifted quadrangle are of the form $x_0$, $\gamma x_0$, $\gamma\phi_j^lx_0$, $\phi_i^kx_0$ for some $\gamma\in \Gamma$ with $\pi(\gamma)=\bar\gamma$. This immediately yields the required conclusion.
 \end{proof}

 Note that $\bar\Gamma$ is also hyperbolic as it is a finite-index subgroup of $\bar G$. In view of the Claim, by \cite[Theorem 7.11]{bowditch_relatively_2012} we have that $\bar\Gamma$ is hyperbolic relative to $\{E( \bar\phi_i)\}$, so that using the relatively hyperbolic Dehn filling theorem \cite{osin_peripheral_2007,groves_dehn_2008} there exists $K\in \Z_{>0}$ such that for all choices $\{n_i\}$ of positive integers we can construct a hyperbolic quotient $\bar{\bar{\Gamma}}$ of $\bar\Gamma$ where the image of $\bar\phi_i$ has order $Kn_i$. We then conclude omnipotence from residual finiteness of the quotients $\bar{\bar{\Gamma}}$.
\end{proof}

\bibliographystyle{plain}

\Addresses

\end{document}